\documentclass[11pt]{article}
\usepackage[utf8]{inputenc}
\usepackage[english]{babel}
\usepackage[a4paper, margin=2.5cm]{geometry}
\usepackage{amsmath, amsthm, amssymb, amsfonts}
\usepackage{mathtools}
\usepackage[shortlabels]{enumitem}
\usepackage{hyperref}
\usepackage[font={small, it}]{caption}
\usepackage{subcaption}
\usepackage{graphicx}
\usepackage{color}
\usepackage{xcolor}
\usepackage{thmtools}
\usepackage{bm}
\usepackage{tikz}
\usepackage{stmaryrd}
\usepackage{mathabx}
\usepackage{bbm}

\newcommand{\cP}{\ensuremath{\mathcal P}}

\newcommand{\eps}{\varepsilon}
\renewcommand{\rho}{\varrho}

\newcommand{\Gnp}{G(n, p)}

\newcommand{\ccl}{\ensuremath{\mathrm{ccl}}}
\newcommand{\vol}{\ensuremath{\mathrm{vol}}}

\declaretheorem[parent=section]{theorem}
\declaretheorem[sibling=theorem]{lemma}

\declaretheorem[sibling=theorem]{question}

\declaretheorem[sibling=theorem]{corollary}

\setlist{itemsep=0.1em, topsep=0.1em}

\hypersetup{
    colorlinks,
    linkcolor={red!60!black},
    citecolor={green!50!black},
    urlcolor={blue!80!black}
}

\colorlet{RoyalRed}{red!70!black}
\definecolor{RoyalBlue}{rgb}{0.25, 0.41, 0.88}
\definecolor{RoyalAzure}{rgb}{0.0, 0.22, 0.66}

\date{}
\title{Complete minors in graphs without sparse cuts}

\author{
	Michael Krivelevich
	\thanks{
		Sackler School of Mathematical Sciences, Tel Aviv University, Tel Aviv 6997801, Israel. Email: \texttt{krivelev@tauex.tau.ac.il}. Research supported in part by USA-Israeli BSF grant 2014361 and by ISF grant 1261/17.
	}
	\and
	Rajko Nenadov	
	\thanks{
		Department of Mathematics, ETH Zurich, Switzerland. Email: {\tt rajko.nenadov@math.ethz.ch}. Research supported in part by SNSF grant 200021-175573.
	}
}

\begin{document}
\maketitle

\begin{abstract}
We show that if $G$ is a graph on $n$ vertices, with all degrees comparable to some $d = d(n)$, and without a sparse cut, for a suitably chosen notion of sparseness, then it contains a complete minor of order
\[
	\Omega\left( \sqrt{\frac{n d}{\log d}} \right).
\]
As a corollary we determine the order of a largest complete minor one can guarantee in $d$-regular graphs for which the second largest eigenvalue is bounded away from $d/2$, in $(d/n, o(d))$-jumbled graphs, and in random $d$-regular graphs, for almost all $d = d(n)$.
\end{abstract}

\section{Introduction}

We say that a graph $G$ contains a graph $H$ as a \emph{minor} ($H \prec G$ for short) if we can obtain $H$ from $G$ by a series of contractions of edges and deletions of vertices and edges of $G$. Equivalently, $H \prec G$ if there are $v(H)$ disjoint subsets $\{T_h\}_{h \in V(H)}$ of $V(G)$ such that each $T_h$ induces a connected graph and there is an edge between $T_h$ and $T_{h'}$ for every $\{h, h'\} \in E(H)$. The \emph{contraction clique number} of $G$, denoted by $\ccl(G)$, is defined as the largest integer $r \in \mathbb{N}$ such that $K_r \prec G$. 

Due to the importance of minors in graph theory it is natural to expect abundance of various results and conjectures providing sufficient conditions for the existence of large complete minors. One such example is Hadwiger's conjecture from 1943 which states $\ccl(G) \ge \chi(G)$, where $\chi(G)$ denotes the chromatic number of $G$. If true, this would be a far reaching generalisation of the four-colour theorem \cite{appel1976every}. So far it has only been resolved in the case where $\chi(G) \le 6$ (see, e.g., \cite{toft1996survey}). For larger values of $\chi$ not even the weaker bound $\ccl(G) = \Omega(\chi(G))$ is known to be true. Currently the best bound is of order $\ccl(G) = \Omega(\chi(G) / \sqrt{\log (\chi (G)})$, obtained independently by Kostochka \cite{kostochka1984lower} and by Thomason \cite{thomason1984extremal}. Bollob\'as, Catlin, and Erd\H{o}s \cite{bollobas80almost} showed that if there exists a counter-example to the Hadwiger's conjecture then it has to be \emph{atypical}, that is the conjecture holds for almost all graph. More precisely, they showed that $G(n,1/2)$, a binomial random graph with edge probability $1/2$, with high probability satisfies $\ccl(G(n,1/2)) = \Theta(n / \sqrt{\log n})$ whereas it is known that $\chi(G(n,1/2)) = (1 \pm o(1)) n / (2 \log n)$ (see, e.g., \cite{frieze2015introduction}). 

Other examples include results on large complete minors in graphs with large girth \cite{diestel2004dense,krivelevich2009minors,kuhn2003minors,thomassen1983girth}, $K_{s,t}$-free graphs \cite{krivelevich2009minors,kuhn2004complete}, graphs without small vertex separators \cite{alon1990separator,kawarabayashi2010separator,plotkin1994shallow}, lifts of graphs \cite{drier2006minors}, random graphs \cite{fountoulakis2008order} and random regular graphs \cite{fountoulakis2009minors}, and others. Krivelevich and Sudakov \cite{krivelevich2009minors} studied the complete minors in graphs with good vertex expansion properties. We complement this by investigating a connection between the contraction clique number of a graph and its edge expansion properties. It turns out that, in some cases, this is the right parameter to look at: as a straightforward corollary of Theorem \ref{thm:main} we determine the order of magnitude of a largest clique minor one can guarantee in classes of random and pseudo-random graphs, for which a result from \cite{krivelevich2009minors} falls short by a factor of $O(\sqrt{\log n})$. Another advantage over vertex expansion is that often edge expansion is easier to verify. 

Let $G = (V, E)$ be a graph with $n$ vertices. Given a subset $S \subseteq V$, the \emph{edge-expansion} of $S$ is defined as
\[
	h(S) = \frac{e(S, V \setminus S)}{|S|}.
\]	
In other words, $h(S)$ denotes the average number of edges a vertex in $S$ sends across the cut. The \emph{Cheeger constant} $h(G)$ of $G$ is then defined as the smallest edge expansion over all subsets of size at most $n/2$: 
\[
	h(G) = \min\{ h(S) \colon S \subseteq V, |S| \le n/2\}.
\]
It is a standard exercise to show that every graph with $e$ edges admits a nearly balanced cut which contains at most $e(1/2 + o(1))$ edges, thus $h(G) \le d(1/2 + o(1))$ for any graph $G$ with average degree $d$. However, it could happen that in highly unbalanced cuts $(S,V \setminus S)$ we actually have a stronger expansion. To capture this, for $1 \le k \le n/2$ we introduce a \emph{restricted} Cheeger  constant $h_{k}(G)$ defined as follows:
\[
	h_k(G) = \min\{ h(S) \colon S \subseteq V, |S| \le k\}.
\]
With this notation at hand, we are ready to state our main result.

\begin{theorem} \label{thm:main}
	For every $\eps > 0$ there exist $\beta > 0$ and $n_0, d_0 \in \mathbb{N}$ such that the following holds for all $n \ge n_0$ and $d \ge d_0$. Let $G$ be a graph with $n \ge n_0$ vertices and maximum degree at most $d$. If $h(G) \ge \eps d$ and $h_{\eps n}(G) \ge (1/2 + \eps)d$ then
	\[
		\ccl(G) \ge \beta \sqrt{\frac{nd}{\log d}}.
	\]
\end{theorem}

It should be noted that in Theorem \ref{thm:main} we allow $d$ to depend on $n$. This is also the case in all other stated results.

Using a first-moment calculation, Fountoulakis, K\"uhn, and Osthus \cite{fountoulakis2008order} showed that with high probability
\begin{equation} \label{eq:random_ccl_upper}
	\ccl(\Gnp) \le (1 + o(1)) \sqrt{\frac{n^2 p}{\log (np)}},
\end{equation}
for $C/n \le p < 1/2$.\footnote{$\log$ denotes the natural logarithm.} For such $p$ we have that $\Gnp$ with high probability contains a large induced subgraph that satisfies assumptions of Theorem \ref{thm:main} with, say, $d = 1.1np$ and some parameter $\eps$, thus the bound on $\ccl(G)$ given in Theorem \ref{thm:main} is, in general, optimal up to a constant factor. 

While the second requirement on the expansion in Theorem \ref{thm:main} ($h_{\eps n}(G) \ge (1/2 + \eps)d$) might seem restrictive at first, it will easily be satisfied in all our applications. The role of this assumption will become apparent in the proof. It remains an interesting problem to determine if one can guarantee the same lower bound on $\ccl(G)$ only assuming $h(G) \ge \eps d$, where $d$ is the maximum degree of $G$. Adapting the proof of Theorem \ref{thm:main}, in \cite{krivelevichExpanders} we showed that this is the case when $d$ is a constant. 

\begin{theorem} \label{thm:constant_d}
	For every $\eps > 0$ there exist $\beta > 0$ and $n_0 \in \mathbb{N}$ such that the following holds. Let $G$ be a graph with $n \ge n_0$ vertices and maximum degree at most $d \ge 3$. If $h(G) \ge \eps d$ then
	\[
		\ccl(G) \ge \beta \sqrt{n}.
	\]
\end{theorem}

Note that Theorem \ref{thm:constant_d} is applicable with any $d = d(n)$, however, as already remarked, it is most likely optimal only in the case $d$ is a constant.

\subsection{Applications}

As a straightforward corollary of Theorem \ref{thm:main}, we improve and extend (and re-prove) several results. Previous proofs of some of these results relied on more specific, and difficult to show,  properties of studied graphs.

One attractive corollary of Theorem \ref{thm:main} relates the size of a largest complete minor in a $d$-regular graph to the second largest eigenvalue $\lambda_2$ of its adjacency matrix. This is not surprising knowing that $\lambda_2$ governs the number of edges in a cut (see, e.g., \cite[Theorem 9.2.1]{alon2016probabilistic}).

\begin{corollary} \label{cor:lambda2}
	For every $\eps > 0$ there exist $\beta > 0$ and $n_0 \in \mathbb{N}$ such that the following holds. Let $G$ be a $d$-regular graph with $n \ge n_0$ vertices, for some $d \ge 3$, and let $\lambda_2$ be the second largest eigenvalue of the adjacency matrix of $G$. If $\lambda_2 < (1/2 - \eps)d$, then $\ccl(G) \ge \beta \sqrt{n d / \log d}$.
\end{corollary}
\begin{proof}
	By \cite[Theorem 9.2.1]{alon2016probabilistic} we have that $e(A, B) \ge (d - \lambda_2)|A||B|/n$ for any disjoint subsets $A, B \subseteq V(G)$. Therefore, for any $S \subseteq V(G)$ we have
	\[
		\frac{e(S, V(G) \setminus S)}{|S|} \ge (1/2 + \eps)d \cdot \frac{|S|(n-|S|)}{|S| n} = (1/2 + \eps)d \cdot (1 - |S|/n).
	\]
	For $|S| \le n/2$ this shows $h(S) \ge d/4$, thus $h(G) \ge d/4$. On the other hand, for $|S| \le \eps n / 4$ we have $h(S) > (1/2 + \eps/2)d$, thus $h_{\eps n / 4}(G) \ge (1/2 + \eps/2)d$. The conclusion of the corollary now follows from Theorem \ref{thm:main} if $d \ge d_0$, for some (large) constant $d_0$, or from Theorem \ref{thm:constant_d} otherwise.
\end{proof}

Using known bounds on the likely value of $\lambda_2$ of the adjacency matrix of a random $d$-regular graph \cite{broder1998optimal,krivelevich2001random,tikhomirov2016spectral} (these results bound the second largest \emph{absolute} eigenvalue, which is stronger than what we need) in the case $d \ge d_0$, for sufficiently large constant $d_0$, and a result by Bollob\'as \cite{bollobas1988isoperimetric} on the Cheeger constant of random $d$-regular graphs for $d < d_0$, we immediately obtain the following result from Corollary \ref{cor:lambda2} and Theorem \ref{thm:constant_d}.

\begin{corollary} \label{cor:regular}
	For any $d \ge 3$, a $d$-regular graph $G_d$ chosen uniformly at random among all $d$-regular graphs with $n$ vertices with high probability satisfies
	\[
		\ccl(G) = \Omega\left( \sqrt{\frac{nd}{\log d}} \right).
	\]		
\end{corollary}

This extends a result by Fountoulakis, K\"uhn, and Osthus \cite{fountoulakis2009minors} who showed the same statement for constant $d \ge 3$ (whereas we allow $d$ to be a function of $n$). Optimality of Corollary \ref{cor:regular} can be derived as follows: Calculations in \cite{fountoulakis2008order} show that with probability at least $1 - \exp(- C n \log (np))$ we have $\ccl(G(n,p)) = O(\sqrt{n^2p / \log (np)})$, for a constant $C$ of our choice (having an impact on the hidden constant in $O(\cdot)$). For $d = d(n) < n/2$, the number of (labelled) $d$-regular graphs with $n$ vertices is of order 
\[
	\Theta\left(  (q^{q} (1 - q)^{1 - q})^{\binom{n}{2}} \binom{n-1}{d}^n \right)
\]
where $q = d / n$ (see \cite{liebenau2017asymptotic,mckay1990asymptotic,mckay1991asymptotic}). A simple calculation using Stirling's approximation shows that by taking $p = d / n$ the random graph $\Gnp$ is $d$-regular with probability at least $\exp(- C' n \log d)$, for some absolute constant $C' > 0$. Therefore we conclude that a random $d$-regular graph $G_d$ satisfies $\ccl(G_d) = O(\sqrt{nd/\log d})$ with high probability.





As our last application we mention a problem of estimating the contraction clique number in \emph{jumbled} graphs, first studied by Thomason \cite{thomason1987pseudo}. Let $G$ be a graph with $n$ vertices and let $p = p(n) \in (0,1)$ and $\beta = \beta(n) > 1$. We say that $G$ is \emph{$(p, \beta)$-jumbled} if for every subset $X \subseteq V(G)$ we have
\[
	\left| e(X) - p\binom{|X|}{2} \right| \le \beta |X|,
\]	
where $e(X)$ denotes the number of edges of $G$ with both endpoints in $X$. 
Krivelevich and Sudakov \cite{krivelevich2009minors} showed that every $(p, o(np))$-jumbled graph contains a complete minor of order
\begin{equation} \label{eq:random_ccl}
	\ccl(\Gnp) = \Omega\left( \sqrt{\frac{n^2 p}{\log (n \sqrt{p})}} \right),
\end{equation}
and the question of whether this bound can be improved to $\Omega(\sqrt{n^2 p / \log (np)}$ was raised in \cite{fountoulakis2009minors}. Note that this matches the bound in \eqref{eq:random_ccl} in case $p = n^{c}$ for a constant $0 < c < 1$, while for $p = C/n$ it falls short by a factor of $\sqrt{\log n}$. Here we settle it in the affirmative for all $p = \Omega(1/n)$.

\begin{corollary}
	Let $G$ be a $(p, o(np))$-jumbled graph with $n$ vertices, for some $C / n \le p \le 1$ where $C$ is a sufficiently large constant. Then 
	\[
		\ccl(G) = \Omega\left( \sqrt{\frac{n^2 p}{\log (np)}} \right).
	\]		
\end{corollary}
\begin{proof}
	By \cite[Lemma 6.1]{krivelevich2009minors}, every $(p, o(np))$-jumbled graph $G$ contains an induced subgraph $G'$ with $n' = (1 - o(1))n$ vertices such that
	\[
		e(S, V(G') \setminus S) \ge (1 - o(1)) p |S| (n' - |S|)
	\]
	for every $S \subseteq V(G')$. Therefore $G'$ satisfies the requirement of Theorem \ref{thm:main} for, say, $\eps = 0.1$. 
\end{proof}

\section{Preliminaries}

We use standard graph-theoretic notation. In particular, given a graph $G$ and a vertex $v \in V(G)$, we denote by $N(v)$ the neighbourhood of $v$. Given a subset $S \subseteq V(G)$, we abbreviate with $N(S)$ the \emph{external} neighbourhood of $S$, that is
\[
	N(S) = \left( \bigcup_{v \in S} N(v) \right) \setminus S.
\]
The distance between two vertices in $G$ is defined as the length (number of edges) of a shortest path between them (thus every vertex is at distance 0 from itself). We say that a subset of vertices $S \subseteq V(G)$ is \emph{connected} if it induces a connected subgraph of $G$.

We start by relating the Cheeger constant to the vertex expansion of a graph and, as a corollary, give a lower bound on the size of a ball around a subset of vertices.

\begin{lemma} \label{lemma:expansion}
	Let $G$ be a graph with $n$ vertices and maximum degree $d = d(n)$. Then for every subset $X \subseteq V(G)$ of size $|X| \le n/2$ we have
	\[
		|N_G(X)| \ge h(G)|X| / d.
	\]		
\end{lemma}
\begin{proof}
	By the definition of $h(G)$ there are at least $|X| h(G)$ edges between $X$ and $V(G) \setminus X$. The desired inequality follows from the fact that every vertex in $V(G) \setminus X$ is incident to at most $d$ of these edges. 
\end{proof}

The following lemma shows that the ball around a subset of vertices grows exponentially until it expands to at least half of the vertex set. 
\begin{lemma} \label{lemma:ball}
	Let $G$ be a graph with $n$ vertices and maximum degree $d = d(n)$. Then for every subset of vertices $U \subseteq V(G)$ and any $i \in \mathbb{N}$, the set
	\[
		B(U,i) = \{v \in V(G) \colon v \text{ is at distance at most } i \text{ from some } u \in U\}
	\]
	is of size at least 
	\[
	|B(U, i)| \ge \min(n/2, |U|(1 + h(G)/d)^i).
	\]
\end{lemma}
\begin{proof}
	Let $U \subseteq V(G)$ be an arbitrary subset. We prove the lower bound on $B(U, i)$ by induction. The claim trivially holds for $i = 0$. Suppose it holds for some $i \ge 0$. If $|B(U,i)| \ge n/2$ then $|B(U,i+1)| \ge n/2$ as well, in which case we are done. Otherwise, if $|B(U,i)| \le n/2$ then we can apply Lemma \ref{lemma:expansion} to conclude $|N(B(U,i))| \ge |B(U,i)| h(G) / d$. The desired bound now follows from $B(U, i + 1) = B(U, i) \cup N(B(U, i))$.
\end{proof}

The following lemma shows that one can find a subset of $V(G)$ of prescribed size that has significantly larger external neighbourhood  than the one given by Lemma \ref{lemma:expansion}, and moreover induces a connected subgraph. The proof and the latter use of the lemma are inspired by a similar statement from \cite{krivelevich2009minors}.

\begin{lemma} \label{lemma:conn_expand}
	Let $\eps > 0$, and suppose $G$ is a graph with $n$ vertices and maximum degree $d = d(n) \ge 2$. If $h_{\eps n}(G) > 1$ then for every integer $1 \le s \le \eps n / (2d)$ and a vertex $v \in V(G)$, there exists a connected subset $X \subseteq V(G)$ of size $|X| = s$ which contains $v$ and
	\[
			|N_G(X)| \ge |X|(h_{\eps n}(G) - 1).
	\]		
\end{lemma}
\begin{proof}
	We prove the claim by induction on $s$. For $s = 1$ we take $X = \{v\}$. As $\delta(G) \ge h_{\eps n}(G)$, $X$ satisfies the desired property. Suppose now that the statements holds for some $1 \le s \le \eps n / (2d) - 1$. Let $X'$ be one such subset of size $|X'| = s$ and $Y = N_G(X')$ be its neighbourhood. If $|Y| \ge h_{\eps n}(G) (s + 1)$ then we can take $X$ to be the union of $X'$ and an arbitrary vertex in $Y$. Otherwise, from $h_{\eps n}(G) \le d$ we have $|X' \cup Y| < \eps n$, thus there are  
	\[
		e(X' \cup Y, V(G) \setminus (X' \cup Y)) \ge |X' \cup Y| h_{\eps n}(G) > |Y| h_{\eps n}(G)
	\]		
	edges between $X' \cup Y$ and the rest of the graph. As none of these edges is incident to $X'$, there exists a vertex $w \in Y$ such that $|N_G(w) \setminus (X' \cup Y)| \ge h_{\eps n}(G)$. Adding such a vertex to $X'$ gives a desired set $X$.
\end{proof}

\subsection{Random walks}

A \emph{lazy random walk} on a graph $G = (V, E)$ with the vertex set $V = \{1, \ldots, n\}$ is a Markov chain whose matrix of transition probabilities $P = P(G) = (p_{i,j})$ is defined by
\[
	p_{i,j} = \begin{cases}
		\frac{1}{2\deg_G(i)}, &\text{if } \{i,j\} \in E(G) \\
		\frac{1}{2}, &\text{if } i = j, \\
		0, &\text{otherwise.}
	\end{cases}
\]
In other words, if at some point we are at vertex $i$ then with probability $0.5$ we stay in $i$ and with probability $0.5$ we move to a randomly chosen neighbour. It is easy to verify (and is well known) that this Markov chain has the stationary distribution $\pi$ given by $\pi(i) = \deg_G(i) / 2e(G)$. The following lemma gives an upper bound on the probability that a lazy random walk avoids some subset $U \subseteq V(G)$.

\begin{lemma} \label{lemma:random_walk_miss_U}
	Let $G$ be a graph with $n$ vertices and maximum degree $d = d(n)$. Then for any $U \subseteq V(G)$ the probability that a lazy random walk on $G$ which starts from the stationary distribution $\pi$ and makes $\ell$ steps does not visit $U$ is at most 
	\[
		\exp\left(- \frac{h(G)^3}{8d^3} \cdot \frac{|U|\ell}{n}\right).
	\]		
\end{lemma}

For the rest of this section we prove Lemma \ref{lemma:random_walk_miss_U}.

Let $\lambda_1 \ge \lambda_2 \ge \ldots \ge \lambda_n$ be eigenvalues of the transition matrix $P$. The \emph{spectral gap} of $P$ is defined as $\delta_\lambda = \lambda_1 - \lambda_2 = 1 - \lambda_2$. The following result of Mossel et al.\ \cite{mossel2006non} (more precisely, the first case of \cite[Theorem 5.4]{mossel2006non}) relates the spectral gap to the probability that a lazy random walk does not leave a specific subset. We state a version tailored to our application.

\begin{theorem} \label{thm:spectral}
	Let $G$ be a connected graph with $n$ vertices and let $\delta_\lambda$ be the spectral gap of the transition matrix $P = P(G)$. Then the probability that a lazy random walk of length $\ell$ which starts from a vertex chosen according to the stationary distribution $\pi$ does not leave a non-empty subset $A \subseteq V(G)$ is at most
	\[
		\pi(A) (1 - \delta_\lambda(1 - \pi(A)))^\ell.
	\]		
\end{theorem}

The second ingredient is a result of Jerrum and Sinclair \cite[Lemma 3.3]{sinclair1989approximate} which relates the spectral gap of $P(G)$ to its \emph{conductance} $\Phi(G)$, defined as
\[
	\Phi(G) = \min_{\substack{S \subseteq V \\ 0 < \pi(S) \le 1/2}} \frac{\sum_{i \in S, j \notin S} \pi(i) p_{i,j}}{\pi(S)} = \min_{\substack{S \subseteq V \\ 0 < \pi(S) \le 1/2}} \frac{e(S, V \setminus S)}{2 \sum_{v \in S} \deg(v)}.
\] 
Note that $\Phi(G) \ge h(G)  / (2d)$, where $d = d(G)$ denotes the maximum degree of $G$.

\begin{lemma} \label{lemma:conductance}
	Let $G = (V, E)$ be a connected graph. Then the spectral gap $\delta_\lambda$ of $G$ is at least $\delta_\lambda \ge \Phi(G)^2/2$.
\end{lemma}

We are now ready to prove Lemma \ref{lemma:random_walk_miss_U}.

\begin{proof}[Proof of Lemma \ref{lemma:random_walk_miss_U}]
	Consider some non-empty subset $U \subset V(G)$. Theorem \ref{thm:spectral} states that a lazy random walk never leaves the set $A = V(G) \setminus U$ with probability at most
	\begin{align} 
		\pi(A)(1 - \delta_\lambda(1 - \pi(A)))^{\ell} &= (1 - \pi(U))(1 - \delta_\lambda \pi(U))^{\ell} \nonumber \\
		&\le \exp(- \delta_\lambda \ell \pi(U)) \le \exp\left(- \frac{h(G)^2}{8 d^2} \cdot \ell \pi(U) \right), \label{eq:U}
	\end{align}
	where in the last inequality we used Lemma \ref{lemma:conductance} and $\Phi(G) \ge h(G)/(2d)$. Note that the events 'leave $A$' and 'visit $U$' are the same. From a trivial bound $2e(G) \le d n$ and $\delta(G) \ge h(G)$ we get
	\[
		\pi(U) \ge \frac{|U| h(G)}{2e(G)} \ge |U| \frac{h(G)}{dn},
	\]		
	which after plugging into \eqref{eq:U} gives the desired probability that a random walk misses $U$.
\end{proof}

\section{Proof of Theorem \ref{thm:main}}
\label{sec:main_proof}

The proof of Theorem \ref{thm:main} combines an approach of Plotkin, Rao, and Smith \cite{plotkin1994shallow} with an idea of Krivelevich and Sudakov \cite{krivelevich2009minors} to use random walks to find  connected subsets of $V(G)$ with some desired properties. That being said, the main new ingredient is the following lemma. 

\begin{lemma} \label{lemma:covering}
	For every $\eps \in (0, 1/2)$ there exist positive $K = O(1/\eps^3)$ and $n_0 \in \mathbb{N}$ such that the following holds. Let $G$ be a graph with $n \ge n_0$ vertices, maximum degree at most $d = d(n)$, and $h(G) \ge \eps d$. Given $s = s(n)$ and $q = q(n) \le n$ such that $sq \ge 2 n$, and subsets $U_1, \ldots, U_q \subseteq V(G)$ where each $U_i$ is of size $|U_i| \ge s$, there exists a connected set $T \subseteq V(G)$ of size at most 
	\[
		|T| \le K \cdot \frac{n}{s} \log \left( \frac{q s}{n} \right)
	\]
	which intersects every $U_i$.
\end{lemma}

We mention in passing that, for many pairs of values of $s(n)$ and $q(n)$, by choosing subsets $U_i\subset [n]$ of size $s$ at random one can see that there is a family $\{U_i\}_{i=1}^q$
whose covering number has order of magnitude $(n/s)\log(qs/n)$. Thus Lemma \ref{lemma:covering} delivers a nearly optimal promise of the size of a hitting set, with an additional -- and for us very important -- benefit of this set inducing a connected subgraph in $G$.

\begin{proof}
	Let
	\[
		\ell = \frac{16}{\eps^3} \cdot \frac{n}{s} \log \left( \frac{qs}{n} \right)\,,
	\]		
	and consider a lazy random walk $W$ in $G$ which starts from the stationary distribution and makes $\ell$ steps. A desired connected subset $T$ is constructed by taking the union of $W$  with a shortest path between $U_j$ and $W$, for each $j \in [q]$. We argue that with positive probability $T$ is of required size.

	\begin{figure}[h!]
		\begin{center}	
		\includegraphics[scale=0.45]{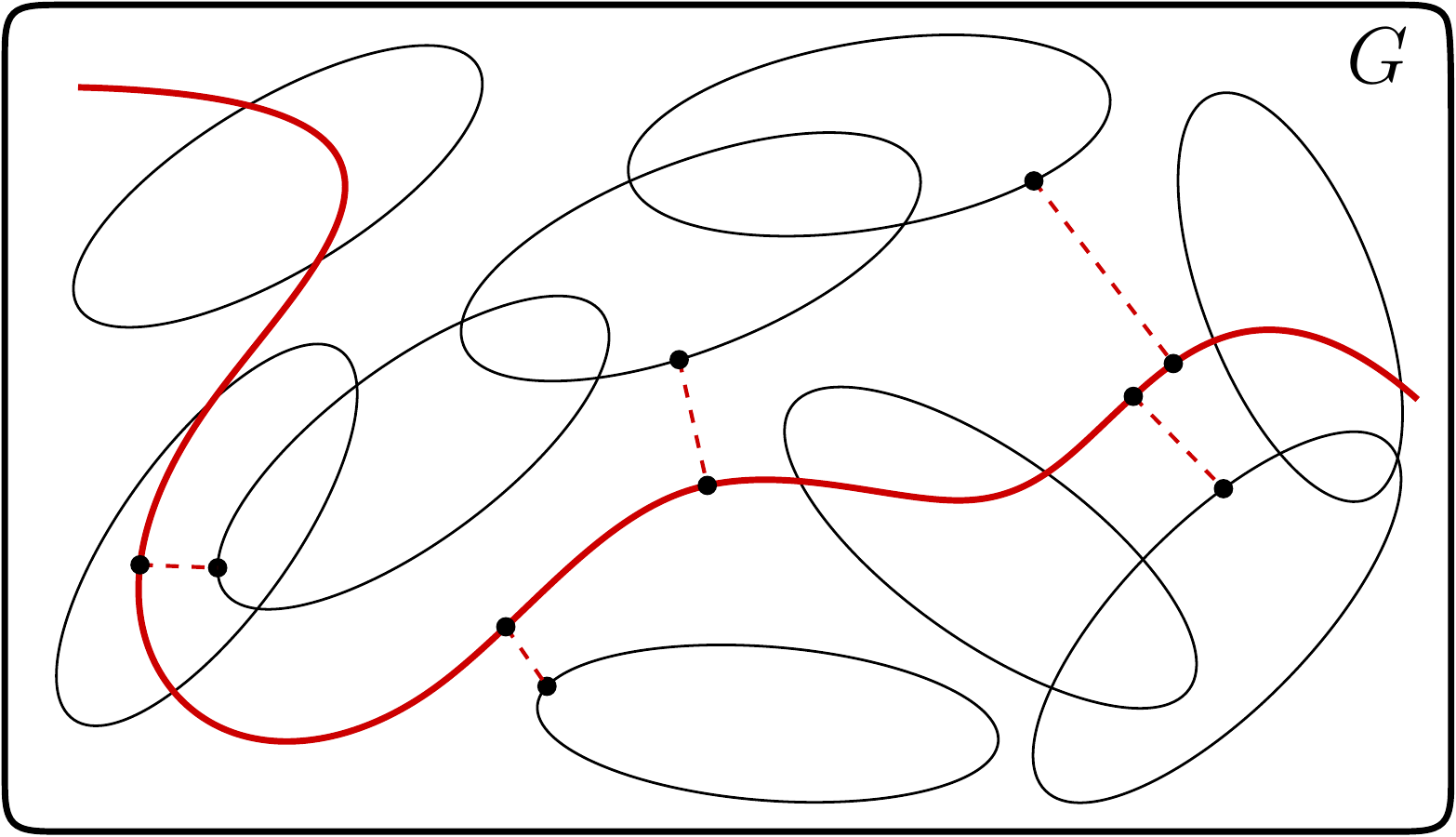}		
		\caption{A random walk in $G$ (the bold path) together with a shortest path (dashed paths) to every subset.}		
		\label{fig:random_walk}		
		\end{center}
	\end{figure}

For $1\le j\le q$, let $X_j$ be the random variable measuring the distance from $U_j$ to $W$ in $G$. Then the set $T$ has expected size at most $\ell + 1 + \sum_{j=1}^q E[X_j]$.

In order to estimate the expectation of $X_j$, for a positive integer $z$ write $p_{j,z}=Pr[X_j=z]$ and $p_{j,\ge z}=Pr[X_j\ge z]$. Trivially $p_{j,z}=p_{j,\ge z}-p_{j,\ge z+1}$ and, as $G$ is connected, $p_{j, n} = 0$. Hence
\[
E[X_j]=\sum_{z\ge 1}^{n} zp_{j,z}=\sum_{z\ge 1}^n z\left(p_{j,\ge z}-p_{j,\ge z+1}\right)\le \sum_{z\ge 1}^n p_{j,\ge z}\,.
\]

By Lemma \ref{lemma:ball} we have that $B(U_j, z)$, the ball of radius $z$ around $U_j$, is of size at least $\min(n/2, s(1 + \eps)^z)$. As $X_j \ge z$ is equivalent to the event that $W$ misses $B(U_j, z-1)$, using Lemma \ref{lemma:random_walk_miss_U} we get the following estimate for $1 \le z \le z_0 := 1 + \log(n/2)/\log(1+\eps)$:
\[
p_{j,\ge z}\le \exp\left\{-\frac{\eps^3(1+\eps)^{z-1} s\ell}{8n}\right\} = \left( \frac{qs}{n} \right)^{2(1 + \eps)^{z-1}}\,.
\]
This implies $p_{j, \ge z'} \le 2^{-n}$ for every $z' \ge z_0$. Straightforward calculation shows:
\[
E[X_j] \le \sum_{z\ge 1}^n p_{j,\ge z}=O\left(\left(\frac{n}{qs}\right)^{2}\right)\,,
\]
hence, finally,
\[
E[|T|]=\ell+\sum_{j=1}^q E[X_j]=\ell+O\left(q\left(\frac{n}{qs}\right)^{2}\right)= \ell + O(n / s) = O(\ell)\,.
\]
\end{proof}


The proof of Theorem \ref{thm:main} splits into three cases, depending on the value of $d$ and $\eps$: the dense case $d \ge \gamma n$, for some $\gamma = \gamma(\eps)$, the intermediate case $\gamma n > d \ge n^{0.5}$, and the most difficult sparse case $d < n^{0.5}$. In the first two cases we shall make use of a classical result obtained independently by Kostochka \cite{kostochka1984lower} and by Thomason \cite{thomason1984extremal}, which states that for any graph $G$ with average degree $d'$ we have $\ccl(G) = \Omega(d' / \sqrt{\log (d')})$. Remarkably, Thomason \cite{thomason2001extremal} has later determined the correct constant in the leading term.

Let us say a few words about the similarities and differences between the sparse and intermediate case. In both cases we follow almost the same arguments, however with different goals. As already said, in the intermediate case we do not directly show that $G$ contains a large complete minor, but rather only a minor of some graph $H$ with $m = \Omega(\sqrt{nd})$ vertices and linear (in $m$) average degree. Applying the Kostochka-Thomason result on $H$ gives a complete minor of order $\sqrt{nd / \log n} = \Theta(\sqrt{nd / \log d})$. Note that this is the same strategy as taken by Krivelevich and Sudakov \cite{krivelevich2009minors}, with the main technical ingredient being Lemma \ref{lemma:random_walk_miss_U}. However, in the sparse case such a bound falls short of the desired one, and we have to show directly that $G$ contains a large complete minor. The most important difference is that instead of Lemma \ref{lemma:random_walk_miss_U} we use Lemma \ref{lemma:covering}. 

For technical reasons our proof of the sparse case fails to work past $d = n / \log n$, thus to avoid worrying about exact calculations we choose $n^{0.5}$ as the delimiter between the two cases. Since, as already mentioned, the arguments are quite similar we only spell out the details of the sparse case and then describe the necessary changes for the intermediate one.

\begin{proof}[Proof of Theorem \ref{thm:main}.] 
	The actual value of $\gamma > 0$ is not important and it will be clear that the calculations work if it is sufficiently small.

	\paragraph{Dense case $d \ge \gamma n$.} From $h(G) \ge \eps d = \Theta(n)$ we conclude that $G$ has linear minimum (and thus average) degree. By the Kostochka-Thomason result it contains a complete minor of order $\Omega(n / \sqrt{\log n})$, as desired.

	\paragraph{Sparse case $d < n^{0.5}$.}
	Let $\zeta = \eps/8$ and let $K$ be a constant given by Lemma \ref{lemma:covering} (for $\eps := \zeta$). Set
	\[
		t = \sqrt{\frac{2 K}{\zeta^2} \cdot \frac{n \log (\zeta^3 \eps d)}{d}} = \Theta\left(\sqrt{\frac{n \log d}{d}} \right) \qquad \text{ and } \qquad r = \frac{\zeta^2 \eps n}{2t} = \Theta\left( \sqrt{\frac{nd}{\log d}} \right).
	\]
	Here $r$ denotes the size of the complete minor we aim to find and $(1 + \zeta)t$ is an upper bound on the size of each subset $T_h$ in the witness for such $K_r$-minor (recall the definition of minors from the very beginning of the paper).

	Let $\cP$ be the family of all ordered partitions $V(G) = D \cup T_1 \cup \ldots \cup T_q \cup U$ which satisfy  the following constraints:
	\begin{enumerate}[(a)]				
		\item \label{p:connected} for each $i \in [q]$ we have $|T_i| \le (1 + \zeta)t$, $T_i$ induces a connected subgraph of $G$, and $|N(T_i)| \ge t (1/2 + 2\zeta)d$, 
		\item \label{p:minor} all $T_i$'s are mutually disjoint and between every $T_i$ and $T_j$ there exists an edge,
		\item \label{p:D_size} $|D| \le 2n/3$, and		
		\item \label{p:D_prop} either $|D| \le \eps n$ and $e(D, U) \le  |D| (1/2 + 3 \zeta)d$, or $|D| > \eps n$ and $e(D, U) \le 3\zeta |D| d$.
	\end{enumerate}
	Taking $D = \emptyset$, $q = 0$ and $U = V(G)$ trivially satisfies all the conditions, thus $\cP$ is non-empty. 

	Let $V(G) = D \cup T_1 \cup \ldots \cup T_q \cup U$ be a partition in $\cP$ which maximises $|D|$ and, among all such partitions, one which further maximises $q$. 
	We show that then necessarily $q \ge r$, which by \ref{p:connected} and \ref{p:minor} gives a witness for $K_r \prec G$. Suppose towards a contradiction that $q < r$.

	We first rule out $D \ge \zeta \eps n$. If $\zeta \eps n \le |D| \le \eps n$ then by the first part of \ref{p:D_prop} we have $e(D, U) \le |D|(1/2 + 3 \zeta) d$. From $h_{\eps n}(G) \ge (1/2 + 8 \zeta) d$ we obtain 
	\[
		e(D, T_1 \cup \ldots \cup T_q) = e(D, V \setminus D) - e(D, U) \ge \zeta |D| d,
	\]
	with room to spare, thus by the assumption that $G$ has maximum degree $d$ we get a contradiction:
	\[
		(1 + \zeta)tr \ge |T_1 \cup \ldots \cup T_q| \ge \zeta^2 \eps n.
	\] 
	Otherwise, if $|D| \ge \eps n$ then from $|D| \le 2n/3$ and $h(G) \ge 8 \zeta d$ we get
	\[
		e(D, V \setminus D) \ge \min\left\{ |D|, n - |D| \right\} h(G) \ge \frac{|D|}{2} 8 \zeta d = |D| 4 \zeta d.
	\]
	Therefore $e(D, T_1 \cup \ldots \cup T_q) \ge \zeta |D| d$ and a contradiction follows as in the previous case.

	For the rest of the proof we assume $|D| \le \zeta \eps n$. Let us collect some further properties of the chosen partition. First, note that for every $i \in [q]$ we have
	\begin{equation} \label{eq:T_i_nbr}
		|N(T_i) \cap U| \ge \zeta td.
	\end{equation}  
	Indeed, if this was not the case then from the property \ref{p:connected} we get
	\begin{equation} \label{eq:reason}
		e(T_i, U) \le e(T_i, V \setminus T_i) - |N(T_i) \setminus U| \le |T_i|d - (t(1/2 + 2\zeta)d - \zeta td) <  |T_i| d / 2.
	\end{equation}
	By adding $T_i$ to $D$ (recall that $t = o(n)$, thus the resulting set is smaller than $\eps n$), relabelling $\{T_j\}_{j \in [q] \setminus \{i\}}$ as $\{T_1, \ldots, T_{q-1}\}$ and decreasing $q$ we obtain a partition which satisfies the desired constraints and has a larger set $D$, contradicting the fact that the chosen partition maximises $|D|$.

	Second, we verify that the subgraph $G_U = G[U]$ of $G$ induced by $U$ satisfies $h(G_U) \ge \zeta d$ and $h_{\eps n / 2}(G_U) \ge (1/2 + 3 \zeta)d$. This is, again, a matter of routine calculations. If there would exist $X \subseteq U$ of size $|X| \le \eps n / 2$ such that 
	\[
		e(X, U \setminus X) < |X| (1/2 + 3\zeta) d,
	\]
	then by removing such $X$ from $U$ and adding it to $D$ we obtain a new partition with the set $D$ of size at most $\eps n$ which clearly satisfies all the constraints, again contradicting the choice of the partition. Similarly, if there exists $X \subseteq U$ of size $\eps n / 2 \le |X| \le |U|/2$ such that 
	\begin{equation} \label{eq:bad_X_U}
		e(X, U \setminus X) < |X| \zeta d,
	\end{equation}
	then from $|D| \le \zeta \eps n$ we get
	\[
		e(D \cup X, U \setminus X) \le e(D, V \setminus D) + e(X, U \setminus X) < \zeta \eps n d  + |X| \zeta d < 3 \zeta |X| d \le 3 \zeta (|X| + |D|) d,
	\]
	which implies we can remove $X$ from $U$ and add it to $D$. Note that the new set $D$ has size at most $\zeta \eps n + |U|/2 \le 2n/3$, with room to spare, again yielding a contradiction.
	
	Having these properties at hand we are ready to proceed towards the final contradiction with the maximality of $q$. From $q < r$ and $|D| \le \zeta \eps n$ we get $|U| \ge n/2$. As $h(G_U) \ge \zeta d$ we can apply Lemma \ref{lemma:covering} on $G_U$ and $U_i = N(T_i) \cap U$ for $i \in [q]$ (we can take $s = \zeta td$ by \eqref{eq:T_i_nbr}). If $sq < 2|U|$ then we add sufficiently many `dummy' sets of size $s$. Let $T \subseteq U$ be the obtained set of size
	\begin{equation} \label{eq:small_T}
		|T| \le K \frac{|U|}{s} \log \left( \frac{qs}{|U|} \right) < K \frac{n}{\zeta td} \log\left( \frac{2rs}{n} \right) < \zeta t.
	\end{equation}
	For every $i \in [q]$ we have that $T$ is disjoint from $T_i$ and there is an edge between them. The only thing that prevents us from declaring $T$ as the new set $T_{q+1}$ in the partition (after removing it from $U$) is that we are not able guarantee $|N(T)| \ge (1/2 + 2\zeta)td$. In fact, as we have chosen $T$ to be rather small this cannot possibly be. However, even if had chosen $T$ to be of size $t$ such a property would not necessarily hold. To fix this, we use that $|T| < \zeta t$ leaves us plenty of space to extend it to a superset $T_{q+1}$ which satisfies this inequality. By Lemma \ref{lemma:conn_expand} $h_{\eps n / 2}(G_U) \ge (1/2 + 3\zeta)d$, there exists a connected subset $T' \subseteq U$ of size $t = o(n/d)$ such that $T'$ contains some vertex from $T$ and 
	\[
		|N(T') \cap U| \ge t\left( (1/2 + 3\zeta)d - 1 \right) \ge t(1/2 + 2.5\zeta)d.
	\]
	It is worth noting that $t = o(n/d)$ is the only place so far where we used an upper bound on $d$, and the previous inequality is the only place where we used $d \ge d_0$ (for some sufficiently large $d_0 = d_0(\zeta)$). 

	The set $T_{q+1} = T \cup T'$ is still disjoint from all the sets $T_i$ for $i \in [q]$ and satisfies
	\[
		|N(T_{q+1})| \ge |N(T_{q+1}) \cap U)| \ge |N(T') \cap U)| - |T| \ge (1/2 + 2\zeta)td.
	\]
	By removing $T_{q+1}$ from $U$ and increasing $q$ we obtain a new partition which by construction satisfies all the constraints, the set $D$ did not change and $q$ has increased; thus a contradiction. This proves that $q \ge r$. 

	\emph{Remark:} Application of Lemma \ref{lemma:conn_expand} explains how we benefit from stronger expansion properties of small subsets. Namely, $h_{\eps n/2}(G_U) \ge (1/2 + 3\zeta)d$ and Lemma \ref{lemma:conn_expand} gave us a connected set $T'$ which has large external neighbourhood, and in particular large enough satisfy the property \ref{p:connected}. The property \ref{p:connected} is necessary for the equation \eqref{eq:reason} to hold which, in turn, shows that \eqref{eq:T_i_nbr} holds. Without the bound from \eqref{eq:T_i_nbr} we would not be able to take $s = \Theta(td)$ which would, finally, result in a larger set $T$ in \eqref{eq:small_T} than what we need.

	\paragraph{Intermediate case $n^{0.5} \le d \le \gamma n$.} We borrow an idea from the proof of \cite[Theorem 4.4]{krivelevich2009minors}: in order to obtain a desired bound on $\ccl(G)$ it suffices to show that $G$ contains a minor of a graph $H$ with $r = \Theta(\sqrt{nd})$ vertices and average degree at least $0.1 r$. By the Kostochka-Thomason result such a graph $H$ contains a complete minor of order $\sqrt{nd / \log n}$ thus, as $\log n = \Theta(\log d)$, we obtain the desired bound.

	We show that $G$ contains a minor of such a graph $H$ by modifying the proof of the sparse case as follows. Instead of asking in \ref{p:minor} that between every $T_i$ and $T_j$ there exists an edge, we ask that there are, say, $0.1 q^2$ pairs $(T_i, T_j)$ for which this holds. Consequently, for this it suffices to take $t = \Theta(\sqrt{n/d})$.

	The argument now remains the same until the point where we seek a contradiction with the maximality of $q$. The crucial observation is that we do not need to find a set $T$ in $U$ which intersects every $N(T_i) \cap U$, but only a $0.1$-fraction of them. A simple application of Lemma \ref{lemma:random_walk_miss_U} gives such a set of size $O(\sqrt{n/d})$. Exactly the same as in the proof of the sparse case, we extend such $T$ to $T_{q+1}$ by finding a connected subset $T'$ which has good vertex expansion properties. Here it is crucial that $t$ is much smaller than $n/d$, which we can achieve by choosing $\gamma$ to be sufficiently small. We leave the details to the reader.
\end{proof}

The proof of Theorem \ref{thm:constant_d} follows the same lines as the presented proof, thus we give a short sketch of the necessary changes. For the full proof see \cite{krivelevichExpanders}.

\begin{proof}[Proof of Theorem \ref{thm:constant_d} (sketch)]
	Define both $t$ and $r$ to be of order $\sqrt{n}$. In the definition of $\mathcal{P}$, we drop the requirement $|N(T_i)| \ge t(1/2 + 2\zeta)d$ in \ref{p:connected}, and in \ref{p:D_prop} we only require $e(D, U) < \eps |D|d$. The rest of the argument proceeds in the same way, with some small further changes. In particular, we replace \eqref{eq:T_i_nbr} by a weaker condition
	\[
		e(T_i,  U) > \eps t d.
	\]
	If this is not true then we can safely move $T_i$ to $D$ and $e(D, U) < \eps |D|d$ would remain satisfied. Such a lower bound on the number of edges and maximum degree imply $|N(T_i) \cap U| \ge \eps t$, which allows us to take only $s = \Theta(t)$, instead of $s = \Theta(td)$. Lemma \ref{lemma:covering} then gives us a set $T$ of size 
	\[
		O\left (\frac{n}{t} \log \frac{rs}{n} \right) = O(\sqrt{n}),
	\]
	and we completely omit the use of Lemma \ref{lemma:conn_expand} as we no longer require $T_{q+1}$ to have (exceptionally) large external neighbourhood.
\end{proof}

\section{Concluding remarks}

We showed that a good edge expander contains a large complete minor. As a corollary, we determined the size of a largest complete minor in random $d$-regular graphs and some families of pseudo-random graphs. These results extend and improve some of the previously known results in a unified way. 

Without any doubt, Theorem \ref{thm:main} would be aesthetically more appealing if the same order of $\ccl(G)$ would hold without the stronger edge expansion assumption on (very) unbalanced cuts. Of course, having the present proof of Theorem \ref{thm:main} in mind one can also come up with different assumptions which would work. For example, $h(G) \ge \eps d$ and every set of size $O(\sqrt{n \log d / d})$ sends at most $o(\sqrt{nd \log d})$ edges to every set of size $\Theta(\sqrt{nd \log d})$. From the point of view of given applications such an assumption would suffice, however it feels more artificial compared to $h_{\eps n}(G) > (1/2 + \eps)d$. That being said, we ask the following question.

\begin{question} \label{q:1}
	Let $G$ be a graph with $n \ge n_0$ vertices, maximum degree at most $d = d(n) \ge 3$, and $h(G) \ge \eps d$, for some constant $\eps > 0$. Is it true that $\ccl(G) = \Omega \left( \sqrt{nd / \log d} \right)$?	
\end{question}

In other words, Question \ref{q:1} asks if one can strengthen the conclusion of Theorem \ref{thm:constant_d} to match the one of Theorem \ref{thm:main}. 

A possible further direction would be to weaken the assumption that the maximum degree and $h(G)$ are of the same order. Unfortunately, such a weakening is no longer sufficient to guarantee even a minor of order $\sqrt{n}$. This can be seen by the following example. Let $G$ be a complete bipartite graph where one side has $n$ and the other side has $n^{1/3}$ vertices. Even though this graph has a very large Cheeger constant, namely $h(G) = \Theta(n^{1/3})$, it does not contain a complete minor of size $n^{1/3} + 2$: every but at most one set $T_h$ has to contain a vertex from the smaller class. It is worth noting that this is in contrast with the case of vertex expansion. A result of Kawarabayashi and Reed \cite{kawarabayashi2010separator} shows that if a graph $G$ is such that every subset of vertices $S \subseteq V(G)$ of size $n/3 \le |S| \le 2n/3$, where $n$ is the number of vertices of $G$, has the external neighbourhood of size at least $\eps|S|$ for some $\eps > 0$, then $G$ contains a complete minor of order $\Omega(\eps \sqrt{n})$. This being said, it remains an interesting question to determine the correct dependency of $\ccl(G)$ on $h(G)$ and its maximum degree. Some dependency could be retrieved from the proof of Theorem \ref{thm:constant_d}, however we did not try to optimise it.

Another interesting class of graphs studied by Krivelevich and Sudakov \cite{krivelevich2009minors} are $K_{s,t}$-free graphs. While the proof from \cite{krivelevich2009minors} relies on establishing vertex expansion properties through edge expansion, it is plausible that using some of the ideas presented here could simplify their proof. However, as the optimal results in this direction were already obtained in \cite{krivelevich2009minors} and they do not seem to follow from our main theorems used as a black box, we did not pursue this.

We now discuss the algorithmic aspect of the problem and the proof of Theorem \ref{thm:main}. First, some background on the spectral graph theory (for a thorough introduction to the topic, see, e.g., \cite{chungSpectral}). Given a graph $G$, let $h'(G)$ be defined as follows:
\[
	h'(G) = \min_{\emptyset \neq S \subsetneq V(G)} \frac{e(S, V(G) \setminus S)}{\min\{ \vol(S), \vol(V(G) \setminus S) \}},
\]
where $\vol(S) = \sum_{v \in S} \deg(v)$. In case $G$ is a $d$-regular graph we have $h(G) = h'(G) / d$, and it is actually $h'(G)$ that some authors refer to as the Cheeger constant. While we could have stated Theorem \ref{thm:main} in terms of $h'(G)$, we have opted for $h(G)$ for convenience. The famous \emph{Cheeger inequality} for graphs \cite{alon1986eigenvalues,alon1985lambda1,dodziuk1984difference} states that
\begin{equation} \label{eq:cheeger_ineq}
	\lambda(G)/2 \le h'(G) \le \sqrt{2 \lambda(G)},
\end{equation}
where $\lambda(G)$ denotes the second smallest eigenvalue of the normalised Laplacian of $G$. A proof of the right hand side of \eqref{eq:cheeger_ineq} yields a polynomial time algorithm which produces a set $S \subseteq V(G)$ such that $\vol(S) \le \vol(V(G))/2$ and $e(S, V(G) \setminus S) \le \vol(S) \sqrt{2 \lambda(G)}$. 

We can now describe (somewhat informally) a randomized algorithm which, given a graph $G$ satisfying assumptions of Theorem \ref{thm:main}, with high probability succeeds in finding a $K_r$-minor in $G$, for some $r = \Theta(\sqrt{nd / \log d})$. We only do it for the sparse case. The other two cases depend on the algorithmic aspects of the Kostochka-Thomason proof, which goes beyond the scope of this paper. The numbers used in the algorithm are (almost) the same as in the proof. Start with a partition $D = \emptyset$, $q = 0$ and $U = V(G)$ and in each iteration do the following until $q = r$:
\begin{itemize}
	\item If there exists $i \in [q]$ such that $T_i$ has less than $t(1/2 + 2\zeta)d$ neighbours in $U$ then move it to $D$ and decrease $q$; 

	\item If there exists a vertex in $U$ with less than $\zeta d$ neighbours in $U$ then move it to $D$;

	\item If $\lambda(G[U]) \le \zeta^2 / 2$ then we can find a subset $S \subseteq U$ such that $\vol_{G[U]}(S) \le \vol_{G[U]}(U)/2$ and $e(S, U \setminus S) \le \vol(S) \sqrt{2 \lambda(G[U])}$. If $|S| \le |U|/2$ then $e(S, U \setminus S) \le |S| \zeta d$, and otherwise $e(U \setminus S, S) \le |U \setminus S| \zeta d$. In any case, we can efficiently find a subset $S' \subseteq U$ such that $|S'| \le |U|/2$ and $e(S', U \setminus S') < |S'| \zeta d$. Move such $S'$ from $U$ to $D$.

	\item Otherwise, if $\lambda(G[U]) \ge \zeta^2/2$ then by \eqref{eq:cheeger_ineq} we have $h'(G[U]) \ge \zeta^2 / 4$. As the minimum degree of $G[U]$ is at least $\zeta d$ we have $\vol_{G[U]}(S) \ge \zeta |S| d$ for every $S \subseteq U$, thus $h(G[U]) \ge \zeta h'(G[U]) d \ge \zeta^3 d / 4$. This is somewhat weaker than what we had in the proof, but it nonetheless suffices. 

	We can now apply the algorithm described in Lemma \ref{lemma:covering}. Run a lazy random walk of length $\ell$ in $G[U]$ and for each $i \in [q]$ take a shortest path from $N(T_i) \cap U$ to this walk. Let $T$ denote the resulting subset of vertices. The expected size of $T$ is at most, say, $\zeta t/2$, for $t$ chosen to be twice of what we used in the proof, hence with probability at least $1/2$ we have $|T| \le \zeta t$. By repeating this procedure $C \log n$ times, with probability at least $1 - 1/n^C$ we find a desired connected subset $T$ of size at most $\zeta t$.

	Finally, apply the argument from the proof of Lemma \ref{lemma:conn_expand} to either find a subset $T' \subseteq U$ of size $t$ which has external neighbourhood of size $t((1/2 + 2.5\zeta)d$, or a subset $S \subseteq U$ of size $|S| \le \eps n / 2$ such that $e(S, U \setminus S) < (1/2 + 3\zeta)d$. In the former case form a new set $T_{q+1} = T \cup T'$ and increase q. In the latter move $S$ to $D$.
\end{itemize}
The proof of Theorem \ref{thm:main} shows that indeed at some point we either terminate with a failure or we get $q = r$. In each iteration either $U$ decreases or $D$ increases, thus the whole procedure finishes after at most $2n$ iterations. The probability of a failure is therefore at most $1/n^{C}$, for any constant $C$ of our choice. 

\bibliographystyle{abbrv}
\bibliography{references}

\begin{thebibliography}{10}

\bibitem{alon1986eigenvalues}
N.~Alon.
\newblock Eigenvalues and expanders.
\newblock {\em Combinatorica}, 6(2):83--96, 1986.

\bibitem{alon1985lambda1}
N.~Alon and V.~D. Milman.
\newblock $\lambda_1$, isoperimetric inequalities for graphs, and
  superconcentrators.
\newblock {\em Journal of Combinatorial Theory, Series B}, 38(1):73--88, 1985.

\bibitem{alon1990separator}
N.~Alon, P.~Seymour, and R.~Thomas.
\newblock A separator theorem for nonplanar graphs.
\newblock {\em Journal of the American Mathematical Society}, 3(4):801--808,
  1990.

\bibitem{alon2016probabilistic}
N.~Alon and J.~H. Spencer.
\newblock {\em The probabilistic method}.
\newblock John Wiley \& Sons, 2016.

\bibitem{appel1976every}
K.~Appel and W.~Haken.
\newblock Every planar map is four colorable.
\newblock {\em Bulletin of the American Mathematical Society}, 82(5):711--712,
  1976.

\bibitem{bollobas1988isoperimetric}
B.~Bollob{\'a}s.
\newblock The isoperimetric number of random regular graphs.
\newblock {\em European Journal of Combinatorics}, 9(3):241--244, 1988.

\bibitem{bollobas80almost}
B.~Bollob\'as, P.~A. Catlin, and P.~Erd\H{o}s.
\newblock Hadwiger's conjecture is true for almost every graph.
\newblock {\em European Journal of Combinatorics}, 1(3):195--199, 1980.

\bibitem{broder1998optimal}
A.~Z. Broder, A.~M. Frieze, S.~Suen, and E.~Upfal.
\newblock Optimal construction of edge-disjoint paths in random graphs.
\newblock {\em SIAM Journal on Computing}, 28(2):541--573, 1998.

\bibitem{chungSpectral}
F.~R.~K. Chung.
\newblock {\em Spectral graph theory}.
\newblock CBMS Regional Conf. Ser. Math. 92. American Mathematical Society,
  Providence, 1997.

\bibitem{diestel2004dense}
R.~Diestel and C.~Rempel.
\newblock Dense minors in graphs of large girth.
\newblock {\em Combinatorica}, 25(1):111--116, 2004.

\bibitem{dodziuk1984difference}
J.~Dodziuk.
\newblock Difference equations, isoperimetric inequality and transience of
  certain random walks.
\newblock {\em Transactions of the American Mathematical Society},
  284(2):787--794, 1984.

\bibitem{drier2006minors}
Y.~Drier and N.~Linial.
\newblock Minors in lifts of graphs.
\newblock {\em Random Structures \& Algorithms}, 29(2):208--225, 2006.

\bibitem{fountoulakis2008order}
N.~Fountoulakis, D.~K{\"u}hn, and D.~Osthus.
\newblock The order of the largest complete minor in a random graph.
\newblock {\em Random Structures \& Algorithms}, 33(2):127--141, 2008.

\bibitem{fountoulakis2009minors}
N.~Fountoulakis, D.~K{\"u}hn, and D.~Osthus.
\newblock Minors in random regular graphs.
\newblock {\em Random Structures \& Algorithms}, 35(4):444--463, 2009.

\bibitem{frieze2015introduction}
A.~Frieze and M.~Karo{\'n}ski.
\newblock {\em Introduction to random graphs}.
\newblock Cambridge University Press, 2015.

\bibitem{kawarabayashi2010separator}
K.-i. Kawarabayashi and B.~Reed.
\newblock A separator theorem in minor-closed classes.
\newblock In {\em Foundations of Computer Science (FOCS), 2010 51st Annual IEEE
  Symposium on}, pages 153--162. IEEE, 2010.

\bibitem{kostochka1984lower}
A.~V. Kostochka.
\newblock Lower bound of the {H}adwiger number of graphs by their average
  degree.
\newblock {\em Combinatorica}, 4(4):307--316, 1984.

\bibitem{krivelevichExpanders}
M.~Krivelevich.
\newblock Expanders – how to find them, and what to find in them.
\newblock In {\em Surveys in Combinatorics 2019}.
\newblock to appear.

\bibitem{krivelevich2009minors}
M.~Krivelevich and B.~Sudakov.
\newblock Minors in expanding graphs.
\newblock {\em Geometric and Functional Analysis}, 19(1):294--331, 2009.

\bibitem{krivelevich2001random}
M.~Krivelevich, B.~Sudakov, V.~H. Vu, and N.~C. Wormald.
\newblock Random regular graphs of high degree.
\newblock {\em Random Structures \& Algorithms}, 18(4):346--363, 2001.

\bibitem{kuhn2003minors}
D.~K{\"u}hn and D.~Osthus.
\newblock Minors in graphs of large girth.
\newblock {\em Random Structures \& Algorithms}, 22(2):213--225, 2003.

\bibitem{kuhn2004complete}
D.~K{\"u}hn and D.~Osthus.
\newblock Complete minors in ${K}_{s, s'}$-free graphs.
\newblock {\em Combinatorica}, 25(1):49--64, 2004.

\bibitem{liebenau2017asymptotic}
A.~Liebenau and N.~Wormald.
\newblock Asymptotic enumeration of graphs by degree sequence, and the degree
  sequence of a random graph.
\newblock {\em arXiv preprint arXiv:1702.08373}, 2017.

\bibitem{mckay1990asymptotic}
B.~D. McKay and N.~C. Wormald.
\newblock Asymptotic enumeration by degree sequence of graphs of high degree.
\newblock {\em European Journal of Combinatorics}, 11(6):565--580, 1990.

\bibitem{mckay1991asymptotic}
B.~D. McKay and N.~C. Wormald.
\newblock Asymptotic enumeration by degree sequence of graphs with degrees
  $o(n^{1/2})$.
\newblock {\em Combinatorica}, 11(4):369--382, 1991.

\bibitem{mossel2006non}
E.~Mossel, R.~O'Donnell, O.~Regev, J.~E. Steif, and B.~Sudakov.
\newblock Non-interactive correlation distillation, inhomogeneous {M}arkov
  chains, and the reverse {B}onami-{B}eckner inequality.
\newblock {\em Israel Journal of Mathematics}, 154(1):299--336, 2006.

\bibitem{plotkin1994shallow}
S.~Plotkin, S.~Rao, and W.~D. Smith.
\newblock Shallow excluded minors and improved graph decompositions.
\newblock In {\em Proceedings of the fifth annual ACM-SIAM symposium on
  Discrete algorithms}, pages 462--470. Society for Industrial and Applied
  Mathematics, 1994.

\bibitem{sinclair1989approximate}
A.~Sinclair and M.~Jerrum.
\newblock Approximate counting, uniform generation and rapidly mixing {M}arkov
  chains.
\newblock {\em Information and Computation}, 82(1):93--133, 1989.

\bibitem{thomason1984extremal}
A.~Thomason.
\newblock An extremal function for contractions of graphs.
\newblock In {\em Mathematical Proceedings of the Cambridge Philosophical
  Society}, volume~95, pages 261--265. Cambridge University Press, 1984.

\bibitem{thomason1987pseudo}
A.~Thomason.
\newblock Pseudo-random graphs.
\newblock In {\em North-Holland Mathematics Studies}, volume 144, pages
  307--331. Elsevier, 1987.

\bibitem{thomason2001extremal}
A.~Thomason.
\newblock The extremal function for complete minors.
\newblock {\em Journal of Combinatorial Theory, Series B}, 81(2):318--338,
  2001.

\bibitem{thomassen1983girth}
C.~Thomassen.
\newblock Girth in graphs.
\newblock {\em Journal of Combinatorial Theory, Series B}, 35(2):129--141,
  1983.

\bibitem{tikhomirov2016spectral}
K.~Tikhomirov and P.~Youssef.
\newblock The spectral gap of dense random regular graphs.
\newblock {\em The Annals of Probability}, 47(1):362--419, 2019.

\bibitem{toft1996survey}
B.~Toft.
\newblock A survey of {H}adwiger's conjecture.
\newblock {\em Congressus Numerantium}, pages 249--283, 1996.

\end{thebibliography}

\end{document}